\documentclass{amsart}

\usepackage{amsmath,amssymb,amsthm}

\usepackage{enumerate}

\newtheorem{thm}{Theorem}[section]
\newtheorem{cor}[thm]{Corollary}
\newtheorem{prop}[thm]{Proposition}
\newtheorem{lem}[thm]{Lemma}
\newtheorem*{lov}{Asymmetric Lov{\'a}sz Local Lemma, Infinite Version}
\theoremstyle{definition}
\newtheorem{defn}[thm]{Definition}
\newtheorem{q}[thm]{Question}

\renewcommand{\leq}{\leqslant}
\renewcommand{\geq}{\geqslant}

\newcommand{\converges}{\mathord{\downarrow}}
\newcommand{\diverges}{\mathord{\uparrow}}

\newcommand{\sub}[1]{_{\textup{\tiny{\fontfamily{cmr}\selectfont #1}}}}

\newcommand{\uhr}{\upharpoonright}

\DeclareMathOperator{\pr}{Pr}
\DeclareMathOperator{\vbl}{vbl}
\DeclareMathOperator{\dom}{dom}
\DeclareMathOperator{\subt}{Subt}

\title[The reverse mathematics of HT for sums of exactly two
elements]{The reverse mathematics of Hindman's Theorem for sums of
exactly two elements} 

\author[Csima]{Barbara F. Csima}
\address{Department of Pure Mathematics, University of Waterloo}
\email{csima@uwaterloo.ca}

\author[Dzhafarov]{Damir D. Dzhafarov}
\address{Department of Mathematics, University of Connecticut}
\email{damir@math.uconn.edu}

\author[Hirschfeldt]{Denis R. Hirschfeldt}
\address{Department of Mathematics, University of Chicago}
\email{drh@math.uchicago.edu}

\author[Jockusch]{Carl G. Jockusch, Jr.}
\address{Department of Mathematics, University of Illinois}
\email{jockusch@math.uiuc.edu}

\author[Solomon]{Reed Solomon}
\address{Department of Mathematics, University of Connecticut}
\email{david.solomon@uconn.edu}

\author[Westrick]{Linda Brown Westrick}
\address{Department of Mathematics, University of Connecticut}
\email{linda.westrick@uconn.edu}

\thanks{Csima was partially supported by Canadian NSERC Discovery Grant
312501. Dzhafarov was partially supported by grant DMS-1400267 from
the National Science Foundation of the United States and a
Collaboration Grant for Mathematicians from the Simons
Foundation. Hirschfeldt was partially supported by grant DMS-1101458
from the National Science Foundation of the United States and a
Collaboration Grant for Mathematicians from the Simons
Foundation. We thank Jason Bell and Jeff Shallit for very helpful
suggestions that led us to the idea of using the Lov\'asz Local
Lemma. We also thank Ludovic Patey for information on RRT$^2_2$.}

\date{\today}

\begin{document}

\begin{abstract}
Hindman's Theorem (HT) states that for every coloring of $\mathbb N$
with finitely many colors, there is an infinite set $H \subseteq
\mathbb N$ such that all nonempty sums of distinct elements of $H$
have the same color. The investigation of restricted versions of HT
from the computability-theoretic and reverse-mathematical perspectives
has been a productive line of research recently. In particular,
HT$^{\leqslant n}_k$ is the restriction of HT to sums of at most $n$
many elements, with at most $k$ colors allowed, and HT$^{=n}_k$ is the
restriction of HT to sums of \emph{exactly} $n$ many elements and $k$
colors. Even HT$^{\leqslant 2}_2$ appears to be a strong principle,
and may even imply HT itself over RCA$_0$. In
contrast, HT$^{=2}_2$ is known to be strictly weaker than HT over
RCA$_0$, since HT$^{=2}_2$ follows immediately from Ramsey's Theorem
for $2$-colorings of pairs. In fact, it was open for several years
whether HT$^{=2}_2$ is computably true.

We show that HT$^{=2}_2$ and similar results with addition replaced by
subtraction and other operations are not provable in RCA$_0$, or even
WKL$_0$. In fact, we show that there is a computable instance of
HT$^{=2}_2$ such that all solutions can compute a function that is
diagonally noncomputable relative to $\emptyset'$. It follows that
there is a computable instance of HT$^{=2}_2$ with no $\Sigma^0_2$
solution, which is the best possible result with respect to the
arithmetical hierarchy. Furthermore, a careful analysis of the proof
of the result above about solutions DNC relative to $\emptyset'$ shows
that HT$^{=2}_2$ implies RRT$^2_2$, the Rainbow Ramsey Theorem for
colorings of pairs for which there are are most two pairs with each
color, over RCA$_0$. The most interesting aspect of our construction
of computable colorings as above is the use of an effective version of
the Lov\'asz Local Lemma due to Rumyantsev and Shen.
\end{abstract}

\maketitle

\section{Introduction}

This paper is concerned with the computability-theoretic and
reverse-math\-e\-mat\-i\-cal analysis of combinatorial principles, in
particular that of versions of Hindman's Theorem, a line of research
that began with the work of Blass, Hirst, and Simpson~\cite{BHS} and
has more recently seen substantial further development. Our main
contribution is to bring to the area the use of probabilistic methods,
in particular the Lov{\'a}sz Local Lemma, in an effective version due
to Rumyantsev and Shen \cite{R,RS}.

We assume familiarity with the basic concepts of computability theory
and reverse mathematics. For a principle $P$ of second-order
arithmetic of the form $\forall X\,[\Theta(X) \rightarrow \exists Y
\Psi(X,Y)]$, an \emph{instance} of $P$ is an $X$ such that
$\Theta(X)$ holds, and a \emph{solution} to this instance is a $Y$
such that $\Psi(X,Y)$ holds.

Hindman's Theorem (HT)~\cite{Hindman} states that for every coloring
$c$ of $\mathbb N$ with finitely many colors, there is an infinite set
$H \subseteq \mathbb N$ such that all nonempty sums of distinct
elements of $H$ have the same color. Blass, Hirst, and
Simpson~\cite{BHS} showed that such an $H$ can always be computed in
the $(\omega+1)$st jump of $c$, and that there is a computable
instance of HT such that every solution computes $\emptyset'$. By
analyzing these proofs they showed that HT is provable in ACA$_0^+$
(the system consisting of RCA$_0$ together with the statement that
$\omega$th jumps exist) and implies ACA$_0$ over RCA$_0$. The exact
reverse-mathematical strength of HT remains open, however.

Recently, there has been interest in investigating restricted versions
of Hindman's Theorem. For instance, HT$^{\leq n}$ is HT restricted to
sums of at most $n$ many elements, and HT$^{=n}$ is HT restricted to sums
of exactly $n$ many elements. We can also consider HT$^{\leq n}_k$ and
HT$^{=n}_k$, the corresponding restrictions to $k$-colorings. (Notice
that HT$^{\leq n}_{k+1}$ clearly implies HT$^{\leq n}_k$, and
similarly for HT$^{=n}$.) An interesting phenomenon is that quite weak
versions of HT are still rather difficult to prove. Indeed, there is
no known way to prove even HT$^{\leq 2}$ other than to give a proof of
the full HT, which has led Hindman, Leader, and Strauss~\cite{HLS} to
ask whether every proof of HT$^{\leq 2}$ is also a proof of HT.

Dzhafarov, Jockusch, Solomon, and Westrick~\cite{DJSW} showed that
HT$^{\leq 3}_3$ implies ACA$_0$ over RCA$_0$ and that HT$^{\leq 2}_2$
is not provable in RCA$_0$. Carlucci, Ko\l{}odzieczyk, Lepore, and
Zdanowski~\cite{CKLZ} investigated versions of Hindman's Theorem for
sums of bounded length in which the solutions are required to meet a
certain natural sparseness condition known as apartness, and in particular
showed that HT$^{\leq 2}_2$ with apartness implies ACA$_0$ over
RCA$_0$. (They also have results for HT$^{=n}_k$ with apartness.) They
then deduced that HT$^{\leq 2}_4$ (with no extra conditions) implies
ACA$_0$ over RCA$_0$. It remains open whether either of HT$^{\leq
  2}_2$ and ACA$_0$ implies the other over RCA$_0$.

The principle HT$^{=2}$ is quite different, as it follows immediately
from Ramsey's Theorem for pairs. For a set $S$, let $[S]^n$ be the set
of $n$-element subsets of $S$. Recall that RT$^n_k$ is the statement
that every $k$-coloring of $[\mathbb N]^n$ has an infinite homogeneous
set, that is, an infinite set $H$ such that all elements of $[H]^n$
have the same color. Then HT$^{=n}_k$ follows at once from RT$^n_k$,
as HT$^{=n}_k$ is essentially the restriction of RT$^{=n}_k$ to
colorings of $[\mathbb N]^n$ where the color of a set depends only on
the sum of its elements. For $n \geq 3$ (and $k \geq 2$), RT$^n_k$ is
equivalent to ACA$_0$ over RCA$_0$, but RT$^2_k$ is a weaker
principle, incomparable with WKL$_0$ (see
e.g.\ Hirschfeldt~\cite{Hbook} for further details).

The possibility was left open in late drafts of~\cite{DJSW} that
HT$^{=2}_2$ is so weak as to be computably true, although a brief note
at the end of the published version mentions the solution of this
problem and more in the current paper. We will show that HT$^{=2}_2$
is not computably true, and indeed, there is a computable instance of
HT$^{=2}_2$ such that the degree of any solution is DNC relative to
$\emptyset'$ (see Section~\ref{results} for a definition). It follows
that this instance does not have any computable or even $\Sigma^0_2$
solutions. Thus HT$^{=2}_2$ is not provable in RCA$_0$, or even in
WKL$_0$. (That HT$^{=2}_2$ does not imply WKL$_0$ follows from the
analogous fact for RT$^2_2$, proved by Liu~\cite{Liu}.) Our method
will also apply to a wider class of principles generalizing
HT$^{=2}_2$, including one studied by Murakami, Yamazaki, and
Yokoyama~\cite{MYY}.

The basic idea for showing that HT$^{=2}_2$ is not computably true is
straightforward. We build a computable instance $c : \mathbb N
\rightarrow 2$ of HT$^{=2}_2$ with no computable solution. Let
$W_0,W_1,\ldots$ be an effective listing of the c.e.\ sets. For $S
\subseteq \mathbb N$ and $s \in \mathbb N$, let $S + s = \{k + s : k
\in S\}$. For each $i$ we choose an appropriately large number $k_i$,
wait until at least $k_i$ many numbers enter $W_i$, and let $E_i$
consist of the first $k_i$ many numbers to enter $W_i$, if $|W_i| \geq
k_i$. We would then like to ensure, for all sufficiently large $s$,
that $E_i + s$ is not homogeneous for $c$, meaning that there are $x,y
\in E_i + s$ such that $c(x) \neq c(y)$. Then $E_i$ cannot be
contained in a solution to $c$, hence in particular $W_i$ cannot be
such a solution.

If we consider only a single fixed $i$, it is easy to define
(uniformly in $i$) a computable coloring $c_i$ that satisfies the
above, i.e., such that for all sufficiently large $s$, we have that
$E_i + s$ is not homogeneous for $c_i$. To do so, let $k_i = 2$, and
let $d = b - a$, where $E_i = \{a, b\}$ and $a < b$. Then define
$c_i$ recursively as follows. If $E_i$ has not been defined by stage
$s$ or $s < d$, let $c(s) = 0$. Otherwise (so $d$ is known at stage
$s$), let $c(s) = 1 - c(s - d)$. Then for all sufficiently large $s$,
we have that $c(b + s) \neq c(b + s - d) = c(a + s)$, so $E_i + s$ is
not homogeneous for $c$ because it contains $b + s$ and $a + s$.

However, the simple method above can break down even for two values of
$i$, say $i_0$ and $i_1$, at least if we take $k_i = 2$ for $i = i_0,
i_1$. In such a case, it could happen that $E_{i_0} = \{0,1\}$ and
$E_{i_1} = \{0,2\}$. Then for every $2$-coloring $c$ of $\mathbb{N}$
and every sufficiently large $s$, at least one of the three sets
$E_{i_0} + s$, $E_{i_1} + s$, and $E_{i_1} + (s+1)$ is homogeneous for
$c$, since otherwise the colors $c(s), c(s+1)$, and $c(s+2)$ are
pairwise distinct, contradicting the assumption that $c$ is a
$2$-coloring. Hence, for some $j \leq 1$, there are infinitely many
$s$ such that $E_{i_j} + s$ is homogeneous. Even if we increase the
$k_i$'s, overlaps between sets $E_i+s$ for different values of $i$ and
$s$ can cause problems in defining $c$. The only way we know to deal
with more than one value of $i$ is to use some version of the
Lov{\'a}sz Local Lemma as described below.

To implement this idea, we think of the bits $c(k)$ as mutually
independent random variables with the values $0$ and $1$ each having
probability $\frac{1}{2}$. If $E_i$ is large, then the event that $E_i
+ s$ is homogeneous for $c$ has low probability, namely $2^{-|E_i| +
 1}$. Furthermore, the events that $E_i + s$ is homogeneous and that
$E_j + t$ is homogeneous are independent whenever $s$ and $t$ are far
apart enough that $E_i + s$ and $E_j + t$ are disjoint. So what we
need is a theorem saying that when we have events with sufficiently
small probability that are somehow sufficiently independent, then it
is possible to avoid all of them at once. That is exactly what the
Lov{\'a}sz Local Lemma does.

However, this is not enough, because we need $c$ to be
computable. Thus we need an effective version of the Lov{\'a}sz Local
Lemma. Fortunately, such a result has been obtained by Rumyantsev and
Shen~\cite{R,RS}, as we describe in the next section. As we will see in
Section~\ref{results}, this result allows us to show easily that our
desired computable $c$ exists. Indeed, by computably approximating
finite subsets $E_i$ of $\Sigma^0_2$ sets, we will be able not only to
avoid computable solutions to $c$, but also to ensure that all
solutions to $c$ have DNC degree relative to $\emptyset'$.

Murakami, Yamazaki, and Yokoyama~\cite{MYY} defined a class of
principles that includes the principles HT$^{=n}_k$ as special cases.
For a function $f: [\mathbb N]^n \rightarrow \mathbb N$, let RT$^f_k$
be the following statement: For any $c : \mathbb N \rightarrow k$,
there is an infinite set $H \subseteq \mathbb N$ such that if $s,t \in
[H]^n$ then $c(f(s))=c(f(t))$. Let RT$^f$ be the principle $\forall
k\, \textup{RT}^f_k$. Notice that if $f(\{x_0,\ldots,x_{n-1}\})=x_0 +
\cdots + x_{n-1}$ then RT$^f_k$ is just HT$^{=n}_k$. As shown
in~\cite{MYY}, if $f: [\mathbb N]^n \rightarrow \mathbb N$ is a
bijection then RT$^f_k$ is equivalent over RCA$_0$ to RT$^n_k$, and
RT$^n_k$ is also equivalent to the statement that RT$^f_k$ holds for
all $f: [\mathbb N]^n \rightarrow \mathbb N$ (and hence implies
RT$^f_k$ for any particular such $f$).

The reason this definition appears in~\cite{MYY} is that the authors
were considering versions of a principle known as the Ramseyan
Factorization Theorem, and they showed that one of these versions is
equivalent to RT$^{\subt}$ for the function
$\subt(\{x_0,x_1\})=|x_0-x_1|$. They proved that RT$^{\subt}$ implies
B$\Sigma^0_2$ over RCA$_0$, with a proof that also applies to HT$^{=2}$,
and indeed to any HT$^f$ such that the image of an infinite set under
$f$ remains infinite. They left open whether RT$^{\subt}_k$ is
provable in RCA$_0$, implies RT$^2_k$, or is somewhere in between
these extremes.

As we will see, our results hold for RT$^{\subt}_2$ as well, and
indeed for RT$^f_2$ for any function satisfying the following
definition.

\begin{defn}
\label{aldef}
A function $f : [\mathbb N]^2 \rightarrow \mathbb N$ is
\emph{addition-like} if
\begin{enumerate}

\item $f$ is computable,

\item there is a computable function $g$ such that if $y>g(x,n)$ then
$f(\{x,y\})>n$, and

\item there is a $b$ such that for all $x \neq y$, there are at most
$b$ many $z$'s for which $f(\{x,z\})=f(\{x,y\})$.

\end{enumerate}
\end{defn}

We will finish the paper with some open questions, but would like to
highlight the following open-ended one here.

\begin{q}
What further uses do the Lov\'asz Local Lemma and other probabilistic
results have in the reverse-mathematical and computability-theoretic
analysis of combinatorial principles?
\end{q}

One example has already been given by Liu, Monin, and
Patey~\cite{LMP}. Another appears in Cholak, Dzhafarov, Hirschfeldt,
and Patey~\cite{CDHP}.

\section{The Lov{\'a}sz Local Lemma and its computable version}

The Lov{\'a}sz Local Lemma was introduced in Erd{\H o}s and
Lov{\'asz}~\cite{EL}. It is a major tool in obtaining lower bounds for
finite Ramsey numbers. See \cite[Section 4.2]{GRS} for a proof and
some applications of the Lov\'asz Local Lemma. The version that we
need, known as the Asymmetric Lov{\'a}sz Local Lemma, first appeared
in Spencer~\cite{S}. It is usually stated in a finite version, but the
infinite version below follows easily from the finite one by a
compactness argument, as pointed out in Proposition 3 of \cite{RS}.

Let $x_0,x_1,\ldots$ be a sequence of mutually independent random
variables, such that each $x_j$ has a finite range, say
$\{0,\ldots,f(j)\}$. Let $A_0,A_1,\ldots$ be events such that each
$A_j$ depends only on the variables $x_n$ for $n$ in some finite set
$\vbl(A_j)$. Thus each event $A_j$ is a Boolean combination of
statements of the form $x_n = k$ for $n \in \vbl(A_j)$ and $k \leq
f(n)$. We can think of $A_j$ as a finite set $S_j$ of functions with
domain $\vbl(A_j)$ such that if $g \in S_j$ then $g(n) \leq f(n)$ for
all $n \in \vbl(A_j)$. An \emph{assignment} of the $x_n$ is just a
function $h : \mathbb N \rightarrow \mathbb N$ such that $h(n) \leq
f(n)$ for all $n$. This assignment \emph{avoids} $A_j$ if the
restriction of $h$ to $\vbl(A_j)$ is not in $S_j$. Let $N(A_j) =
\{A_t : \vbl(A_t) \cap \vbl (A_j) \neq \emptyset\}$, and assume that
each $N(A_j)$ is finite. Then we have the following.

\begin{lov}
Suppose the above hypotheses hold and there exist $r_0,r_1,\ldots \in
(0,1)$ such that
$$
\pr[A_j] \leq r_j \cdot \prod_{\substack{A_t \in N(A_j) \\ t \neq j}} (1-r_t)
$$
for all $j$. Then there is an assignment of $x_0,x_1,\ldots$ that
avoids every $A_j$.
\end{lov}

Moser and Tardos~\cite{MT} gave an efficient algorithm for finding
such an assignment for $x_0, x_1, \dots, x_{n-1}$ in the finite
version of this theorem. As noted in~\cite{RS}, Fortnow then
conjectured that an effective version of the theorem should also
hold. This conjecture was confirmed as follows.

Let $x_0,x_1,\ldots$ and $A_0, A_1,\ldots$ be as above. Assume that
the function $f$ bounding the ranges of the $x_n$ is computable, and
that the $x_n$ have uniformly computable rational-valued probability
distributions. Assume also that the $A_j$ are uniformly computable
(i.e., that there is a computable procedure that, given $j$, returns
$\vbl(A_j)$ and the set $S_j$ as above). The assumption that each
$N(A_j)$ is finite means that each $n$ is in $\vbl(A_j)$ for only
finitely many $j$. Assume that we have a procedure for computing a
canonical index of this finite set given $j$. The following result is
the effective version of the Lov\'asz Local Lemma, whose proof first
appeared in Rumyantsev~\cite{R} and was subsequently published in
Rumyantsev and Shen~\cite{RS}. Note that the hypothesis of the
effective version is a bit stronger than that of the original version,
as the upper bound on Pr$[A_j]$ in the original version is multiplied
by the factor $q < 1$ to obtain the upper bound in the effective
version.

\begin{thm}[Rumyantsev and Shen~\cite{R,RS}]
\label{RSthm}
Suppose the above hypotheses hold and there are $q \in \mathbb Q \cap
(0,1)$ and a computable sequence $r_0,r_1,\ldots \in \mathbb Q \cap
(0,1)$ such that
$$
\pr[A_j] \leq qr_j \cdot \prod_{\substack{A_t \in N(A_j) \\ t \neq j}} (1-r_t)
$$
for all $j$. Then there is a computable assignment of $x_0,x_1,\ldots$
that avoids every $A_j$.
\end{thm}

The following consequence of this result is a slightly restated
version of one given in~\cite[Corollary 7.2]{RS}. For a finite partial
function $\sigma$, the \emph{size} of $\sigma$ is
$|\dom(\sigma)|$. When we say that a sequence
$\sigma_0,\sigma_1,\ldots$ of finite partial functions is computable,
we mean that there is a computable procedure that, given $i$, returns
$\dom(\sigma_i)$ and the values of $\sigma_i$ on this domain.

\begin{cor}[Rumyantsev and Shen~\cite{RS}] 
\label{rscor}
For each $q \in (0,1)$ there is an $M$ such that the following
holds. Let $\sigma_0,\sigma_1,\ldots$ be a computable sequence of
finite partial functions $\mathbb N \rightarrow 2$, each of size at
least $M$. Suppose that for each $m \geq M$ and $n$, there are at most
$2^{qm}$ many $j$ such that $\sigma_j$ has size $m$ and $n \in
\dom(\sigma_j)$, and that we can computably determine the set of all
such $j$ given $m$ and $n$. Then there is a computable $c : \mathbb N
\rightarrow 2$ such that for each $j$ there is an $n \in
\dom(\sigma_j)$ with $c(n)=\sigma_j(n)$.
\end{cor}

From this result it is easy to conclude the following fact, which is
the one we will use in the the next section. To obtain it, apply
Corollary \ref{rscor} to the sequence $\sigma_0, \sigma_1, \dots$,
where $\sigma_{2j}$ and $\sigma_{2j+1}$ each have domain $F_j$, and
$\sigma_{2j}(x) = 0$ and $\sigma_{2j+1}(x) = 1$ for all $x \in F_j$,
choosing $q$ in Corollary \ref{rscor} to be greater than the given $q$
for Corollary \ref{usecor} below.

\begin{cor}
\label{usecor}
For each $q \in (0,1)$ there is an $M$ such that the following
holds. Let $F_0,F_1,\ldots$ be a computable sequence of finite sets,
each of size at least $M$. Suppose that for each $m \geq M$ and $n$,
there are at most $2^{qm}$ many $j$ such that $|F_j|=m$ and $n \in
F_j$, and that we can computably determine a canonical index for the
set of all such $j$ given $m$ and $n$. Then there is a computable $c :
\mathbb N \rightarrow 2$ such that for each $j$ the set $F_j$ is not
homogeneous for $c$.
\end{cor}

\section{The effective content of HT$^{=2}_2$ and some generalizations}
\label{results}

The next result will be considerably generalized in Theorem
\ref{main}. Nonetheless, we include it here to illustrate an
application of Corollary \ref{usecor} in a simple context. The proof
of Theorem \ref{main} will have the same basic idea but will also
involve computable approximations to $\Sigma^0_2$ sets and
addition-like functions replacing addition.

\begin{thm}
\label{comp}
The principle $\textup{HT}^{=2}_2$ is not computably true. That is,
it has a computable instance with no computable solution.
\end{thm}

\begin{proof}
We follow the outline of the proof given in the introduction. Let $M$
be as in Corollary \ref{usecor} for $q = \frac{1}{2}$, where we assume
without loss of generality that $m \leq 2^{\frac{m}{2}}$ for all $m
\geq M$. For each $i$, let $k_i = M + i$. For each $i$ with $|W_i|
\geq k_i$, let $E_i$ consist of the first $k_i$ many elements
enumerated into $W_i$, and let $E_i$ be undefined if $|W_i| <
k_i$. Let $F_0, F_1, \ldots$ be a computable enumeration without
repetitions of all finite sets of the form $E_i + s$ (over all $i, s
\in \mathbb N$) such that $W_i$ contains at least $k_i$ many elements
by stage $s$ (so that $E_i$ is known by stage $s$). Clearly, if $E_i$
is defined, then for all sufficiently large $s$ the set $E_i + s$
occurs in the sequence $F_0, F_1, \ldots$, and conversely, every set in
the sequence $F_0, F_1, \ldots$ of cardinality $k_i$ has the form $E_i
+ s$ for some $s$.

As explained in the introduction, it suffices to show that Corollary
\ref{usecor} applies to the sequence $F_0, F_1, \ldots$, since this
corollary then gives the existence of a computable coloring $c :
\mathbb N \rightarrow 2$ such that no $F_j$ is homogeneous for $c$. It
follows that for all $i$ with $E_i$ defined, if $s$ is sufficiently
large then $E_i + s$ is not homogeneous for $c$, so no solution to
$c$ can contain $E_i$, and in particular $W_i$ is not a solution to
$c$.

We now verify that the hypotheses of Corollary \ref{usecor} are
satisfied. Let $m \geq M$ and $n$ be given. We claim that there are at
most $m$ many values of $j$ such that $|F_j|=m$ and $n \in F_j$. Let
$i = m - M$, so that $|E_i| = k_i = m$. The claim asserts that there
are most $m$ many values of $s$ such that $E_i + s$ occurs in the
sequence $F_0, F_1, \ldots$ and $n \in E_i + s$. If $n \in E_i + s$,
then $n = x + s$ for some $x \in E_i$. There are $m$ many choices for
$x$ and for each $x$ there is a unique $s$ with $n = x + s$, so the
claim is proved. Since $m \leq 2^{\frac{m}{2}}$ by the choice of $M$,
there are at most $2^{\frac{m}{2}}$ many values of $j$ such that
$|F_j|=m$ and $n \in F_j$. It remains to check that the set of such
$j$ can be effectively computed from $m$ and $n$. Again, let $i = m -
M$. We must effectively compute the canonical index of the set $S$ of
$s$ such that $W_i$ contains at least $m$ many elements by stage $s$
and $n \in E_i + s$. If $n \in E_i + s$, then $s \leq n$. So for each
$s \leq n$ we can check effectively whether $W_i$ contains at least
$m$ many elements by the end of stage $s$. If not, $s \notin S$. If
so, we can effectively compute $E_i$ and then effectively determine
whether $n \in E_i + s$, and hence whether $s \in S$. Hence, we can
apply Corollary \ref{usecor} as described in the previous paragraph.
\end{proof}

A function $f$ is \emph{diagonally noncomputable} (\emph{DNC})
relative to an oracle $X$ if $f(e) \neq \Phi^X_e(e)$ for all $e$ such
that $\Phi^X_e(e)$ is defined, where $\Phi_e$ is the $e$th Turing
functional. A degree is \emph{DNC} relative to $X$ if it computes a
function that is DNC relative to $X$. An infinite set $A$ is
\emph{effectively immune} relative to $X$ if there is an
$X$-computable function $f$ such that if $W^X_e \subseteq A$ then
$|W^X_e| < f(e)$, where $W_e$ is the $e$th enumeration operator.

\begin{thm}[Jockusch~\cite{Jfpf}]
\label{Jthm}
A degree is DNC relative to $X$ if and only if it computes a set that
is effectively immune relative to $X$.
\end{thm}
 
Let $W^{\emptyset'}_0,W^{\emptyset'}_1,\ldots$ be an effective list of
the $\Sigma^0_2$ sets, with corresponding computable approximations
$W^{\emptyset'}_i[s]$ (chosen so that $x \in W^{\emptyset'}_i$ if{}f
for all sufficiently large $s$, we have $x \in
W^{\emptyset'}_i[s]$). We adopt the standard convention that if $x \in
W^{\emptyset'}_i[s]$ then $x < s$. For a function $f : [\mathbb N]^2
\rightarrow \mathbb N$ and $x \neq y$, we write $f(x,y)$ for
$f(\{x,y\})$. For a set $S \not\ni y$, we write $f(S,y)$ for $\{f(x,y)
: x \in S\}$.

It follows from the proof of Theorem \ref{comp} that there is a
computable instance of HT$^{=2}_2$ such that all solutions are
effectively immune relative to $\emptyset$, and hence have degrees
that are DNC relative to $\emptyset$. In the following theorem, which
is our main result, we replace $\emptyset$ by $\emptyset'$ as an
oracle and simultaneously replace addition by an arbitrary
addition-like operation as defined in Definition \ref{aldef}.

\begin{thm}
\label{main}
Let $f$ be addition-like. There is a computable instance of
$\textup{RT}^f_2$ such that the degree of any solution is DNC relative
to $\emptyset'$.
\end{thm}

\begin{proof}
Let $b$ be a constant witnessing that $f$ is addition-like, as in part
(3) of Definition~\ref{aldef}. Note that the fact that $f$ is
addition-like implies that if $F$ is a finite set and $x \notin F$,
then for all but finitely many $y$, we have $\min f(F,y) > \max
f(F,x)$. Let $M$ be as in Corollary~\ref{usecor} for
$q=\frac{1}{2}$. We may assume that $M>0$ and $M$ is sufficiently
large so that $bm^2 \leq 2^{\frac{m}{2}}$ for all $m \geq
M$.

Given $i$ and $s$, for each $x \in W^{\emptyset'}_i[s]$, let $t_x$ be
the least $t$ such that $x \in W^{\emptyset'}_i[u]$ for all $u \in
[t,s]$. (I.e., $t_x$ measures how long $x$ has been in
$W^{\emptyset'}_i$.) Order the elements of $W^{\emptyset'}_i[s]$ by
letting $x \prec y$ if either $t_x<t_y$ or both $t_x=t_y$ and
$x<y$. Let $E_i[s]$ be the set consisting of the least $b(M+i)$ many
elements of $W^{\emptyset'}_i[s]$ under this ordering, or
$E_i[s]=\emptyset$ if $W^{\emptyset'}_i[s]$ has fewer than $b(M+i)$ many
elements.\footnote{This definition could be simplified by noting that
there is a partial $\emptyset'$-computable function $\psi$ such that
if $|W^{\emptyset'}_i| \geq b(M+i)$ then $\psi(i)$ is the canonical
index of a set $E_i \subseteq W^{\emptyset'}_i$ such that
$|E_i|=b(M+i)$. The limit lemma then gives us a computable binary
function $g$ such that $\psi(i)=\lim_s g(i,s)$ for all $i$ such that
$\psi(i)$ is defined, and we can define $E_i[s]$ to be the set with
canonical index $g(i,s)$ if this set has size $b(M+i)$, and
$E_i[s]=\emptyset$ otherwise. However, the current definition will
make it easier to describe the adaptation of this proof to one over
RCA$_0$ in the next section.}

The following properties of this definition are the ones that matter
to us:
\begin{enumerate}

\item The function taking $i$ and $s$ to $E_i[s]$ is computable. 

\item Every element of $E_i[s]$ is less than $s$, so $f(E_i[s],s)$ is
defined.

\item If $E_i[s] \neq \emptyset$ then $|f(E_i[s],s)| \geq M+i$.

\item If $|W^{\emptyset'}_i| \geq b(M+i)$ then there
is a $t$ such that $E_i[t] \neq \emptyset$ and $E_i[s]=E_i[t]
\subseteq W^{\emptyset'}_i$ for all $s \geq t$.

\end{enumerate}

We build a computable sequence of finite sets $F_0,F_1,\ldots$ as
follows. Order the pairs $i,s$ via a standard pairing function, and go
through each such pair in order. If $E_i[s] = \emptyset$ then proceed
to the next pair. Otherwise, let $s_0$ be least such that
$E_i[t]=E_i[s]$ for all $t \in [s_0,s]$. Suppose that the following
hold.
\begin{enumerate}

\item $\min f(E_i[s],s)>s_0$.

\item If $u<s_0$ and $E_i[u] \neq \emptyset$ then $\min f(E_i[s],s) >
\max f(E_i[u],u)$. 

\end{enumerate}
Then add $f(E_i[s],s)$ to our sequence. We say that $f(E_i[s],s)$ was
enumerated into our sequence by $i$. Otherwise do nothing. In
any case, proceed to the next pair.

Notice that if there is an $s_0$ such that $E_i[s]=E_i[s_0] \neq
\emptyset$ for all $s \geq s_0$, then we add $f(E_i[s],s)$ to our
sequence for all sufficiently large $s$, because for each $u<s_0$ and
$n \leq \max f(E_i[u],u)$, there are only finitely many $s$ such that
$n \in f(E_i[s_0],s)$, and similarly for each $n \leq s_0$.

Now $F_0,F_1,\ldots$ is a computable sequence of finite sets, each of
size at least $M$. Suppose that $n \in F_j$ and $|F_j|=m$. Then $F_j$
was enumerated by some $i < m$. (Actually $i \leq m-M$.) If $n$ is
also in $F_l$ and $F_l$ was also enumerated by $i$, then we must have
$F_j=f(E_i[s],s)$ and $F_l=f(E_i[t],t)$ for some $s$ and $t$ such that
$E_i[t]=E_i[s]$. For each $x \in E_i[s]$, there are at most $b$ many
$t$ such that $f(x,t)=n$, so there are at most $bm$ many such
$l$. Thus the total number of elements of size $m$ in our sequence
that contain $n$ is at most $bm^2 \leq 2^{\frac{m}{2}}$.

By part (2) of Definition~\ref{aldef}, given $n$ and $m$, we can
computably determine a stage $s \geq n$ such that for each $i<m$ and
$t \geq s$, we have $\min f(E_i[n],t)>n$. It follows from the
definition of our sequence that if $F$ is enumerated into it a stage
at which we are working with a pair $i,t$ with $i<m$ and $t \geq s$,
then $\min F > n$. So we can compute the set of all $j$ such that
$|F_j|=m$ and $n \in F_j$.

Thus the hypotheses of Corollary~\ref{usecor} are satisfied, and hence
there is a computable $c$ as in that corollary. Suppose that
$|W^{\emptyset'}_i| \geq b(M+i)$. Then there is an $F \subseteq
W^{\emptyset'}_i$ such that $f(F,s)$ is in our sequence for all
sufficiently large $s$. For each such $s$, there are $x,y \in F$ such
that $c(f(x,s)) \neq c(f(y,s))$, so $F$ cannot be contained in a
solution to $c$ as an instance of RT$^f_2$. Thus, if $H$ is a solution
to $c$ and $W^{\emptyset'}_i \subseteq H$, then $|W^{\emptyset'}_i| <
b(M+i)$, which means that $H$ is effectively immune relative to
$\emptyset'$, and so has DNC degree relative to $\emptyset'$.
\end{proof}

The computable instance $c$ constructed above cannot have any
$\Sigma^0_2$ solutions, since no $\Sigma^0_2$ set is effectively
immune relative to $\emptyset'$. Thus we have the following fact,
whose analogs for RT$^2_2$ and HT were proved by Jockusch~\cite{J} and
Blass, Hirst, and Simpson~\cite{BHS}, respectively.

\begin{cor}
Let $f$ be addition-like. There is a computable instance of
$\textup{RT}^f_2$ with no $\Sigma^0_2$ solution.
\end{cor}

In particular, both HT$^{=2}_2$ and RT$^{\subt}_2$ have computable
instances with no $\Sigma^0_2$ solutions. On the other hand, every
computable instance of HT$^{=2}_2$ does have a $\Pi^0_2$ solution
since the corresponding result holds for RT$^2_2$ by~\cite{J}.

Every principle RT$^f_2$ has the form $\forall X\,[\Theta(X)
\rightarrow \exists Y\, (Y \mbox{ is infinite and } \Psi(X,Y))]$
where $\Psi$ is $\Pi^0_1$. Thus we can obtain a further result from
the following general fact.

\begin{lem}
Let $P$ be a principle of the form \[\forall X\,[\Theta(X)
\rightarrow \exists Y\, (Y \mbox{ is infinite and } \Psi(X,Y))]\]
where $\Psi$ is $\Pi^0_1$. Suppose that $P$ has a
computable instance $X$ with no low solution. Then every solution to
$X$ is hyperimmune.
\end{lem}

\begin{proof}
Assume for a contradiction that $X$ has a solution $Y$ that is not
hyperimmune. Let $F_0,F_1,\ldots$ be a computable sequence of pairwise
disjoint finite sets such that $Y \cap F_i \neq \emptyset$ for all
$i$. Let $\mathcal C$ be the collection of all $Z$ such that
$\Psi(X,Z)$ holds and $Z \cap F_i \neq \emptyset$ for all $i$. Then
$\mathcal C$ is a $\Pi^0_1$ class, and is nonempty as it contains
$Y$. By the Low Basis Theorem, $\mathcal C$ has a low element. This
element is a solution to $X$, contradicting the choice of $X$.
\end{proof}

\begin{cor}
Let $f$ be addition-like. There is a computable instance of
$\textup{RT}^f_2$ such that all solutions are hyperimmune.
\end{cor}

\section{The logical strength of HT$^{=2}_2$ and generalizations}

As mentioned above, RT$^2_k$ implies RT$^f_k$ for every $f : [\mathbb
N]^2 \rightarrow k$, but does not imply WKL$_0$. Since WKL$_0$ has an
$\omega$-model consisting entirely of $\Delta^0_2$ sets, we have the
following.

\begin{cor}
Let $f$ be addition-like. Then $\textup{RT}^f_k$ is incomparable with
$\textup{WKL}_0$ over $\textup{RCA}_0$.
\end{cor}

In particular, both HT$^{=2}_k$ and RT$^{\subt}_k$ are incomparable with
$\textup{WKL}_0$ over $\textup{RCA}_0$.

Theorem~\ref{main} also has a reverse-mathematical version. For the
purposes of reverse mathematics, we should alter the definition of
addition-like function to remove the computability requirements. In
other words, $f : [\mathbb N]^2 \rightarrow \mathbb N$ is
addition-like in the sense of reverse mathematics if there is a
function $g$ such that if $y>g(x,n)$ then $f(x,y)>n$, and there is a
$b$ such that for all $x \neq y$, there are at most $b$ many $z$'s for
which $f(x,z)=f(x,y)$.

We also need to be careful in defining the reverse-mathematical analog
of the notion of being DNC over the jump, since the existence of the
jump cannot be proved in RCA$_0$. Given a set $X$, we can of course
approximate $X'$, so we can define $\Phi_e^{X'}(x)[s]$ as
usual. We adopt the convention that if $\Phi_e^{X'}(x)[s]\converges$
with use $u$ and $X'[s+1] \uhr u \neq X'[s] \uhr u$, then
$\Phi_e^{X'}(x)[s+1]\diverges$. We now define $\Phi_e^{X'}(x)=y$ to mean
that $\exists t\, \forall s \geq t\, [\Phi_e^{X'}(x)[s]=y]$. We write
$\Phi_e^{X'}(x) \neq y$ to mean that either $\Phi_e^{X'}(x)\diverges$
or $\Phi_e^{X'}(x)=z$ for $z \neq y$. We write $n \in W_e^{X'}$ to
mean that $\exists t\, \forall s \geq t \, [n \in W_e^{X'}[s]]$, where
$W_i^{X'}[s] = \{n<s : \Phi_e^{X'}(n)[s]\converges\}$.

Now \emph{2-DNC} is the statement that for every $X$, there is a
function $h$ such that $h(e) \neq \Phi_e^{X'}(e)$ for all $e$.

Inspecting the proofs of Theorem~\ref{RSthm} and Corollary~\ref{rscor}
in~\cite{RS}, we see that they can be carried out in RCA$_0$. Thus we
can obtain the following analog of Corollary~\ref{usecor}.

\begin{cor}
The following is provable in $\textup{RCA}_0$: For each $q \in (0,1)$
there is an $M$ such that the following holds. Let $F_0,F_1,\ldots$ be
a sequence of finite sets, each of size at least $M$. Suppose that for
each $m \geq M$ and $n$, there are at most $2^{qm}$ many $i$ such that
$|F_i|=m$ and $n \in F_i$, and that there is a function taking $m$ and
$n$ to the set of all such $i$. Then there is a $c : \mathbb N
\rightarrow 2$ such that for each $i$ the set $F_i$ is not homogeneous
for $c$.
\end{cor}

The proof of Theorem~\ref{main}, relativized to a given oracle $X$,
can now be carried out in RCA$_0$, except for one issue: In the
absence of $\Sigma^0_2$-bounding, it is possible to have $b(M+i)$ many
$n$ such that $n \in W_i^{X'}$ without having a single $s$ such that
$|W_i^{X'}[s]| \geq b(M+i)$. In this case, we would have
$E_i[s]=\emptyset$ for all $s$.

To get around this issue, we do not attempt to establish effective
immunity relative to $X'$, but work instead with a modified
notion. Write $\|W_e^{X'}\| \geq m$ to mean that there are a
finite set $F$ with $|F| \geq m$ and a $t$ such that $n \in
W_e^{X'}[s]$ for all $n \in F$ and $s \geq t$. Now \emph{2-EI} is
the statement that for each $X$, there are an infinite set $A$ and a
function $f$ such that if $\|W_e^{X'}\| \geq f(e)$, then there is
an $n \in W_e^{X'}$ with $n \notin A$.

The proof of Theorem~\ref{main}, relativized to an arbitrary $X$,
shows that if $f$ is addition-like then $\textup{RT}^f_2$ implies
$2$-EI over RCA$_0$. The main point to notice in that proof is the
following: Suppose that $\|W_i^{X'}\| \geq b(M+i)$. By definition,
there are $F$ and $t$ such that $|F| \geq b(M+i)$ and $n \in
W_i^{X'}[s]$ for all $n \in F$ and $s \geq t$. By bounded
$\Pi^0_1$-comprehension, which holds in RCA$_0$, we can form the set
$\widehat{F}$ of all $n \leq \max F$ such that $n \in W_i^{X'}[s]$ for
all $s \geq t$, and then let $G$ be the set consisting of the $b(M+i)$
many least elements of $\widehat{F}$ in the $\prec$-ordering defined
at stage $t$. If $k \in W_i^{X'}[t] \setminus G$ then there is an $s_k
> t$ such that $k \notin W_i^{X'}[s_k]$. By $\Sigma^0_1$-bounding,
which holds in RCA$_0$, there is a $u$ such that we can take $s_k \leq
u$ for all such $k$. If $s \geq u$, then for any $k \in W_i^{X'}[s]
\setminus G$ and any $n \in G$, we have that $n \prec k$ for the
ordering $\prec$ defined at stage $s$. It follows that $E_i[s]=G$ for
$s \geq u$.

To obtain $2$-DNC, we use the following proposition, whose proof is
based on that of Theorem~\ref{Jthm} given in~\cite{Jfpf}. (We need
only one direction of the proposition, but the equivalence it
establishes is of independent interest.)

\begin{prop}
$2\textup{-EI}$ is equivalent to $2\textup{-DNC}$ over
$\textup{RCA}_0$.
\end{prop}

\begin{proof}
We argue in RCA$_0$. First suppose that $2$-EI holds. Given $X$, let
$A$ and $f$ be as in the statement of $2$-EI. Write $W_e^{X'} \approx
W_i^{X'}$ if $W_e^{X'}[s]=W_i^{X'}[s]$ for all sufficiently large
$s$. Notice that in this case, for each $n$ we have $n \in W_e^{X'}$
if{}f $n \in W_i^{X'}$, and $\|W_e^{X'}\| \geq m$ if{}f $\|W_i^{X'}\|
\geq m$.

Let $n_0<n_1<\cdots$ be the elements of $A$ in order. There is a
function $g$ such that $W_{g(e)}^{X'}[s]=\{n_i<s : i<f(e)\}$ for all
$s$. Then $W_{g(e)}^{X'} \not\approx W_e^{X'}$ for all $e$, as otherwise
we would have $\|W_e^{X'}\| \geq f(e)$ but $n \in A$ for all $n \in
W_e^{X'}$.

There is a function $p$ such that $W^{X'}_{p(e)}[s]=W^{X'}_{y}[s]$ if
$\Phi_e^{X'}(e)[s]=y$, and $W^{X'}_{p(e)}[s]=\emptyset$ if
$\Phi_e^{X'}(e)[s]\diverges$. Let $h = g \circ p$. If
$\Phi^{X'}_e(e)=y$ then $W^{X'}_{p(e)} \approx W^{X'}_y$. But
$W_{h(e)}^{X'} \not\approx W_{p(e)}^{X'}$, since $h(e)=g(p(e))$, so
$W^{X'}_{h(e)} \not\approx W^{X'}_y$, and hence $h(e) \neq y$. Thus
$h$ is as in the definition of $2$-DNC.

Now suppose that $2$-DNC holds. Given $X$, let $h$ be as in the
statement of $2$-DNC. We first define a function $g$ such that for
each $e$, we have $g(e) \neq \Phi_i^{X'}(i)$ for all $i \leq e$. Let
$\tau_0,\tau_1,\ldots$ list the elements of $\omega^{<\omega}$. Let
$(\tau)_i$ be the $i$th element of $\tau$ if $|\tau|>i$, and let
$(\tau)_i=0$ otherwise. There is a function $r$ such that
$\Phi^{X'}_{r(i)}(r(i))=(\tau_k)_i$ if $\Phi^{X'}_i(i)=k$ and
$\Phi^{X'}_{r(i)}(r(i))\diverges$ if $\Phi^{X'}_i(i)\diverges$. Let
$g(e)$ be such that $|\tau_{g(e)}|=e+1$ and
$(\tau_{g(e)})_i=h(r(i))$ for all $i \leq e$. If $i \leq e$ and
$\Phi_i^{X'}(i)=k$ then $(\tau_{g(e)})_i \neq (\tau_k)_i$, so $g(e)
\neq k$.

Let $D_0,D_1,\ldots$ list the finite sets. Order the elements of
$W^{X'}_e[s]$ as in the proof of Theorem~\ref{main}. That is, for $x
\in W^{X'}_e[s]$, let $t_x$ be the least $t$ such that $x \in
W^{X'}_e[u]$ for all $u \in [t,s]$, then let $x \prec y$ if either
$t_x<t_y$ or both $t_x=t_y$ and $x<y$. There is a function $q$ such
that if $W^{X'}_e[s] \nsubseteq D_i$, then
$\Phi^{X'}_{q(e,i)}(q(e,i))[s]=n$ for the $\prec$-least $n \in
W^{X'}_e[s] \setminus D_i$.

We now define sequences $a_0<a_1<\cdots$ and $k_0,k_1,\ldots$ as
follows. Suppose that we have defined $a_j$ and $k_j$ for all
$j<e$. Let $i_e$ be such that $D_{i_e}=\{a_0,\ldots,a_{e-1}\}$. Let
$k_e=q(e,i_e)$. For $j \leq a_{e-1}$, let $m_j$ be such that
$\Phi^{X'}_{m_j}(m_j)=j$. Let
$m=\max\{k_0,\ldots,k_e,m_0,\ldots,m_{a_{e-1}}\}$ and let
$a_e=g(m)$. Notice that $a_e>a_{e-1}$.

Let $A=\{a_0,a_1,\ldots\}$. This set exists because the $a_e$ are
defined in order. Now suppose that $\|W_e^{X'}\| \geq e+1$. Then there
are a finite set $F$ with $|F|=e+1$ and a $t$ such that for all $s
\geq t$, every element of $F$ is in $W_e^{X'}[s]$. Let $$S=\{n\leq\max
F : \exists s \geq t\, [n \notin W_e^{X'}[s]]\}.$$ By
$\Sigma^0_1$-bounding, there is a $u \geq t$ such that $\exists s \in
[t,u]\, [n \notin W_e^{X'}[s]]$ for all $n \in S$. Let $G$ consist of
the $e+1$ least elements of $W_e^{X'}[u]$ under the
$\prec$-ordering. If $s \geq u$ then the elements of $G$ are also the
least $e+1$ many elements of $W^{X'}_e[s]$ under the
$\prec$-ordering. Since $|D_{i_e}|=e$, there is an $n \in G$ such that
$\Phi^{X'}_{q(e,i_e)}(q(e,i_e))[s]=n$ for all $s \geq u$, and hence
$\Phi^{X'}_{q(e,i_e)}(q(e,i_e))=n$. By the definition of $q$, we have
that $n \neq a_j$ for $j<e$, and by construction, $n \neq a_j$ for $j
\geq e$. Thus $n \in W^{X'}_e$ but $n \notin A$. So $A$ and the
function $e \mapsto e+1$ are as required by $2$-EI.
\end{proof}

We thus have the following result.

\begin{thm}
$\textup{RCA}_0$ proves that if $f$ is addition-like then
$\textup{RT}^f_2$ implies $2\textup{-DNC}$.
\end{thm}

This theorem can be understood as an implication between
Ramsey-theoretic principles, because Miller [unpublished] has shown
that $2$-DNC is equivalent over RCA$_0$ to RRT$^2_2$, a version of the
Rainbow Ramsey Theorem that states that if $c : [\mathbb N]^2
\rightarrow \mathbb N$ is such that $|c^{-1}(i)| \leq 2$ for all $i$,
then there is an infinite set $R$ such that $c$ is injective on
$[R]^2$.

\begin{cor}
$\textup{RCA}_0$ proves that if $f$ is addition-like then
$\textup{RT}^f_2$ implies $\textup{RRT}^2_2$.
\end{cor}

For those familiar with Weihrauch reducibility, we will also say that
the proofs in~\cite{RS} are uniform, so the $c$ in
Corollary~\ref{usecor} can be obtained uniformly from the sequence
$F_0,F_1,\ldots$ (for a fixed $q$). The proof of Theorem~\ref{main} is
also uniform, as is the proof that computing an effectively immune set
implies computing a DNC function. Thus if $f$ is addition-like (in the
original sense of Definition~\ref{aldef}), then $2\textrm{-DNC}
\leq\sub{W} \textup{RT}^f_2$. Miller's aforementioned argument shows
that $\textrm{RRT}^2_2 \leq\sub{W} 2\textrm{-DNC}$, so we also have
that $\textrm{RRT}^2_2 \leq\sub{W} \textup{RT}^f_2$.

\section{Open Questions}

We finish with some open questions. Implications here could be over
RCA$_0$ or in the sense of notions of computability-theoretic
reduction such as Weihrauch reducibility.

\begin{q}
Does HT$^{=2}_2$ imply RT$^2_2$?
\end{q}

\begin{q}
Does RRT$^2_2$ imply HT$^{=2}_2$? 
\end{q}

\begin{q}
What is the exact relationship between HT$^{=2}_j$ and HT$^{=2}_k$ for
$j \neq k$?
\end{q}

\begin{q}
Does either of HT$^{=2}_2$ and RT$^{\subt}_2$ imply the other?
\end{q}

Jockusch~\cite{J} showed that for each $n \geq 2$, there is a
computable instance of RT$^n_2$ with no $\Sigma^0_n$ solution.

\begin{q}
For $n \geq 3$, is there a computable instance of HT$^{=n}_2$ with no
$\Sigma^0_n$ solution?
\end{q}

A positive answer to this question would imply that HT is not provable
in ACA$_0$, for the same reason that Jockusch's aforementioned result
implies that $\textrm{RT}$ (the principle $\forall n\, \forall k\,
\textrm{RT}^n_k$) is not provable in ACA$_0$ (see e.g.\ Section 6.3 of
\cite{Hbook}). The question is also open for HT$^{\leq n}_2$, and even
for full HT.

\begin{q}
Is it true that for every degree $\mathbf a$ that is DNC relative to
$\emptyset'$ and every computable instance $c$ of HT$^{=2}_2$, there
is an $\mathbf a$-computable solution to $c$?
\end{q}

A positive answer to the above question would show that Theorem
\ref{main} is best possible in a strong sense.

\begin{q}
What is the first-order strength of HT$^{=2}_2$? What about HT$^{=2}$?
\end{q}

Of course, analogs of the above questions can also be asked for
RT$^{\subt}_2$ or RT$^f_2$ for other addition-like (but not bijective)
functions $f$.

\end{document}